\documentclass{article}
\usepackage{amsmath,amssymb,amsthm}
\usepackage{graphicx}
\usepackage{tikz}
\usetikzlibrary{patterns}
\usetikzlibrary{positioning}

\newtheorem{theorem}{Theorem}

\newtheorem{remark}[theorem]{Remark}
\newtheorem{lemma}[theorem]{Lemma}

\numberwithin{theorem}{section}
\numberwithin{equation}{section}

\title{Sectorial Paley-Wiener Theorem}
\author{Armen Vagharshakyan\\
\small\textit{Institute of Mathematics NAS, Armenia}}
\date{}
\begin{document}
\maketitle
\begin{abstract}
By correcting, simplifying and extending a result of Morimoto, we prove a Paley-Wiener type theorem for functions of exponential type in a sector. It serves as a sectorial analogue of Polya's theorem on the indicator of entire functions and improves a result of Dzhrbashyan and Avetisyan by finding the maximal convex set of analytic continuation inside a sector.
\end{abstract}
\section{ Introduction.}
A Paley–Wiener type theorem is any theorem that relates decay properties of a function or distribution at infinity with analyticity of its Fourier transform. The first of such theorems was proved by Paley and Wiener \cite{PW} for certain entire functions of exponential type:
\begin{theorem}[Paley-Wiener]\label{PW}
The set of entire functions $f$ of exponential type at most $h,$ that are square-integrable on the real axis, coincides with the set of functions $f$ that allow the integral representation
\begin{equation*}
    f(z)=\int_{-h}^{h}e^{iz\omega}g(\omega)d\omega,\quad \text{for }g\in L_2(-h,h).
\end{equation*}
\end{theorem}
We prove the following Paley-Wiener-type theorem for functions of exponential type in a sector:
\begin{theorem}\label{spw}
(a) Let $\alpha$ be an angle in the first quadrant,
\begin{equation}\label{alpha}
    0<\alpha<\pi/2,
\end{equation}
let $\Delta\subset\mathbb{C}$ the open sector determined by the angle $\alpha,$ 
\begin{equation}\label{s}
    \Delta:=\lbrace z\in\mathbb{C}\setminus\lbrace 0\rbrace\colon -\alpha<\arg(z)<\alpha\rbrace,
\end{equation}
let the function $f$ be analytic 
and of exponential type at most $h\geq 0$ in $\Delta,$ i.e.
for any $\epsilon>0$ let there exist a constant $C_{\epsilon}\geq 0$ such that
\begin{equation}\label{f_original}
    |f(z)|\leq C_{\epsilon}e^{(h+\epsilon)|z|},\quad \text{ for } z\in \Delta,
\end{equation}
let $f(\zeta),\; \zeta\in\partial \Delta$ denote the non-tangential limit of $f$ at $\zeta$ (see remark \ref{boundary}),
\newline
(b) let $\Gamma$  be the curve parameterized by
\begin{align}\label{gamma}
&
\gamma\colon\mathbb{R}\rightarrow\mathbb{C},\nonumber
\\
&
\gamma(t)=p-ie^{i\alpha}|t|,\quad\text{for }t\in (-\infty,0],\nonumber
\\
&
\gamma(t)=p+ie^{-i\alpha}|t|,\quad\text{for }t\in (0,+\infty],
\end{align}
where $p$ is a real number satisfying
\begin{equation}\label{p}
p\cos(\alpha)<-h
\end{equation}
(see illustration \ref{fig_gamma} on page \pageref{fig_gamma}),
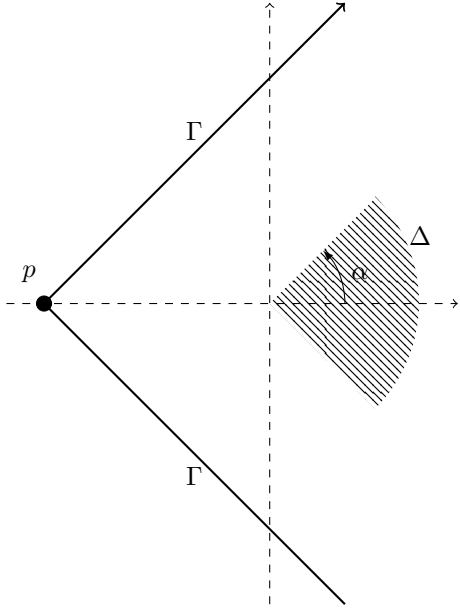
\begin{figure}\label{fig_gamma}
\begin{tikzpicture}[scale=1]
\draw[dashed,->] (-3.5,0) -- (2.5,0); 
\draw[dashed,->] (0,-4) -- (0,4); 
\draw[fill=black] (-3,0) circle (0.1cm); 
\draw[thick,->] (-3,0) -- (1,4);
\draw[thick] (-3,0) -- (1,-4);
\node at (-3.2,0.4) {$p$}; 
\node at (-1,2.3) {$\Gamma$}; 
\node at (-1,-2.3) {$\Gamma$}; 
\fill[pattern=north west lines, pattern color=black]
(0,0) 
-- 
(1.4,-1.4) arc[start angle=-45, end angle=45,radius=2cm] 
-- 
(1.4,1.4)
;
\node at (2,0.9) {$\Delta$}; 
\draw[-latex] (0:1) arc[start angle=0, end angle=45,radius=1cm]  ;
\node at (1.2,0.4) {$\alpha$}; 
\end{tikzpicture}
\caption{ The curve $\Gamma.$}
\end{figure}
\newline
(c) for $|\theta|\leq \alpha$ let $I_f(\theta)$ denote the indicator of $f$ in direction $\theta,$ namely 
\begin{equation}\label{indicator}
    I_f(\theta):=\limsup_{s\rightarrow +\infty}\frac{\ln\left|f\left(se^{i\theta}\right)\right|}{s},
\end{equation}
and correspondingly let $\Omega_{\theta}$ denote the half-plane,
\begin{equation}\label{omega_directional}
    \Omega_{\theta}:=\lbrace \omega\colon Re\left(\omega e^{i\theta}\right)<-I_f(\theta)\rbrace,
\end{equation}
let $g_{\theta}$ be the directional Laplace transform,
defined on the half-plane $\Omega_{\theta}$ by
\begin{equation}\label{g_directional}
g_{\theta}(\omega):=\frac{1}{2\pi i}
\int_{e^{i\theta}[0,+\infty)}
f(\zeta)e^{\omega \zeta}
d\zeta,\quad \text{for }\omega\in\Omega_{\theta},
\end{equation}
and let $g$ be the concatenated Laplace transform of the function $f;$ that is,
$g$ is defined on the union $\Omega$ of all half-planes $\Omega_{\theta},$
\begin{equation}\label{omega}
    \Omega:=\bigcup_{\theta\colon e^{i\theta}[0,+\infty)\subset \bar{\Delta}}
    \Omega_{\theta},
\end{equation}
in terms of directional Laplace transforms $g_{\theta}$ as follows:
\begin{equation}\label{g}
    g(\omega):=g_{\theta}(\omega),\quad \text{for }\omega\in \Omega_{\theta}\subset \Omega.
\end{equation}
(We will prove in lemma \ref{well-defined} that the function $g$ is well-defined by formula \eqref{g}.)
Then the following Fourier inversion formula holds:
\begin{equation}\label{morimoto}
     f(z)=\int_{\Gamma} g(\omega)e^{-\omega z}d\omega.
 \end{equation}
\end{theorem}
The outline of proof of theorem \ref{spw} is as follows: we first express $f$ via a Cauchy integral formula \eqref{cauchy_integral_formula}; then we rewrite the Cauchy integral formula \eqref{cauchy_integral_formula} as a Fourier inversion formula \eqref{fourierinversion}; finally, in remark \ref{fasg} we rewrite the internal integrals appearing in the Fourier inversion formula \eqref{fourierinversion} in a unified way, namely, as values of the concatenated Laplace transform $g.$
\begin{remark}\label{remark_sector}
We remark on the formulation of theorem \ref{spw} that the reasons we take the set $\Delta$ to be a sector are the following:
with each point $e^{i\theta}\in \Delta$ the set $\Delta$ contains the whole ray $e^{i\theta}(0,+\infty),$ so that we can introduce the notion of indicator \eqref{indicator} properly; additionally, the set of angles
$
    \lbrace \theta \colon e^{i\theta}(0,+\infty)\subset\bar{\Delta}\rbrace
$
is connected, so that we can apply the Cauchy integral formula  \eqref{bounded_cauchy_ntegral_formula}.
\end{remark}
\begin{remark}\label{remark_halfplane}
One can eliminate restriction \eqref{alpha} on the angle $\alpha$ and obtain results similar to theorem \ref{spw} for wider, more general sectors $\Delta$ if one splits the sector into narrower ones.
\end{remark}
We complement theorem \ref{spw} by relating the indicator of $f$ to the natural domain of analyticity of $g$ by the following theorem:
\begin{theorem}\label{nda}
Under notations of theorem \ref{spw}, for $|\theta|<\alpha,$ the concatenated Laplace transform $g$ may not be analytically continued to any half-plane $\Omega^{\prime}$  contains the half-plane $\Omega_{\theta}$ properly,
$
     \Omega_{\theta}\subsetneq \Omega^{\prime}.
 $
\end{theorem}
As shown in section \ref{nda_proof}, theorem \ref{nda} follows from theorem \ref{spw}.

This paper was motivated by the following remark: 
\begin{remark}
In \cite{YS} the authors Yoshino and Suwa claim that for a function $f$ of exponential type in the right open half-plane the following Paley-Wiener-type representation holds:
\begin{equation}\label{ys}
    f(z)=\int_{\mathbb{C}}e^{-\omega z}d\mu(\omega),
\end{equation}
where $\mu$ is a measure defined in the complex plane $\mathbb{C},$ whose support is bounded from above, below, and the right.
Their claim is based on proposition 5.2 page 95 of Morimoto's article \cite{M}.
In turn, in lemma 5.2 page 93 of the same article \cite{M} the following is claimed: if a function is holomorphic on the set $W(1)=\lbrace \omega\colon Im(\omega)>k_2+\epsilon\rbrace,$ then it is holomorphic on a larger set 
$\lbrace \omega\colon  Im(\omega)>k_2\rbrace;$
an obviously wrong claim.
This mistake affects the construction of the Laplace transform in definition 5.3 page 95 of \cite{M} and ultimately affects the proof of proposition 5.2 of \cite{M}.

However, a closer inspection of the article \cite{M} shows that under proper conditions \eqref{f_original} imposed on the function $f$ its boundary values exist almost everywhere (see remark \ref{boundary}), and consequently the Laplace transform may be defined explicitly (see formula \ref{g_directional}) instead of having to rely on the elaborate technique of factor spaces of holomorphic functions. 

In this paper we use this latter approach to derive a representation of the function $f$ in terms of an integral of exponents.
\end{remark}
\begin{remark}
After this article was completed, the author came across the article \cite{DA} of Dzhrbashyan and Avetisyan where, in its paragraph 1, theorem \ref{spw} was proved in quite a similar way to our proof, some eighteen years before Morimoto's article \cite{M}. 
\end{remark}
The following two remarks compare our results \ref{spw},\ref{nda} to those of Morimoto \eqref{ys} and Polya \ref{P}.
\begin{remark}
Theorem \ref{spw} extends the result \eqref{ys} in the following ways: our claim holds for a general sector not only for half-plane; and instead of a measure $\mu$ integrated over the complex plane $\mathbb{C},$ we use a function $g$ integrated over a curve $\Gamma;$ 
\end{remark}
\begin{remark}
We extend the result \eqref{ys} in yet another way: we complement theorem \ref{spw} by theorem \ref{nda} regarding the natural domain of analyticity of the function $g.$ In fact theorems \ref{spw},\ref{nda} can be viewed as an analogue of the Polya's theorem (see \cite{P},\cite{B} or theorem 5.5 in \cite{L}) for a sector:
\begin{theorem}[Polya]\label{P}
Let $f$ be an entire function of exponential type.
Denote by $K\subset \mathbb{C}$ the convex set whose support function $k(\theta)$ is determined by $f$'s indicator as follows:
\begin{equation}
    k(\theta)=I_f(\theta)
\end{equation}
Then $f$ can be restored by
\begin{equation}
    f(z)=\frac{1}{2\pi i}\int_{\Gamma}g(t)e^{-zt}dt,
\end{equation}
where $\Gamma$ is a closed contour containing the set $K,$ and
$g$ is the Laplace transform of $f.$ Additionally, $K$ is the smallest convex set such that $g$ is analytic in $\mathbb{C}\setminus K.$
\end{theorem}
In fact, Polya's theorem implies the Paley-Wiener theorem (see \cite{LLST},section 10.1).
\end{remark}
\begin{remark}
The peculiarity of our results is the following: as stated in the very name of article \cite{M}, unlike the setting discussed in the theorems of Paley-Wiener \ref{PW} or Polya \ref{P}, the curve $\Gamma,$ that appears in the Fourier inversion formula \eqref{morimoto}, is unbounded. Additionally, unlike in the case of remark 4.1.2 in \cite{ASP}, we may not express the directional Laplace transform \eqref{g_directional} in terms of $f$'s Taylor series expansion at $0,$ simply because $f$ may not be analytic at $0.$
\end{remark}

Theorem \ref{spw} has proved to have several applications, e.g. in refining Carlson's uniqueness theorem \cite{V} and in proving trigonometric convexity for the multidimensional indicator after Ivanov \cite{MV}.

\section{ Cauchy formula.}
\begin{remark}\label{boundary}
Note that, due to condition \eqref{f_extended}, the function $f$ is bounded in a vicinity of each of the points of $\partial \Delta.$ Hence, by Fatou's theorem, the non-tangential limit of $f$ exists at a.e. point $z\in\partial \Delta.$ 
Extend the function $f$ from the open sector $\Delta$ to the closed sector $\bar{\Delta}$ in the following way: for $z\in\partial \Delta$ define $f(z)$ to be the non-tangential limit provided by Fatou's theorem if the limit exists and define $f(z)$ to be $0$ if the limit doesn't exist. Cauchy integral formula is then applicable to any bounded domain in $\bar{\Delta}.$  Additionally, the extension of the function $f$ to $\bar{\Delta},$ defined in the aforementioned way, satisfies the same estimate \eqref{f_original} on the whole closed sector $\bar{\Delta}:$ 
for any $\epsilon>0$ there exists a constant $C_{\epsilon}\geq 0$ such that
\begin{equation}\label{f_extended}
|f(z)|\leq C_{\epsilon}e^{(h+\epsilon)|z|},\quad \text{ for } z\in \bar{\Delta}.
\end{equation}
\end{remark}
Due to \eqref{f_extended},\eqref{s} and the fact that $p$ is a negative number, we can estimate
\begin{equation}\label{pl}
  \left|f(z)e^{pz}\right|
 \leq
     C_{\epsilon }e^{(h+\epsilon+p\cos(\alpha))|z|},\quad \text{for }z\in\bar{\Delta}.
\end{equation}
In particular, due to the choice \eqref{p} of point $p$ we can take $\epsilon$ small enough to guarantee that the coefficient 
\begin{equation}\label{negative}
h+\epsilon+p\cos(\alpha)
\end{equation}
in the power of the last exponent of \eqref{pl} is negative.
This is why
the Phragmen-Lindeloef maximum principle applies 
to the Cauchy integral formula for the function $f(z)e^{pz}$ in the 
sector $\bar{\Delta}.$ 
Consequently, we have
\begin{equation}\label{bounded_cauchy_ntegral_formula}
    f(z)e^{pz}
    =
    \frac{1}{2\pi i}
    \int_{\partial \Delta}
    \frac{f(\zeta)e^{p\zeta}}{\zeta-z}
    d\zeta,
    \quad \text{for }z\in \Delta,
\end{equation}
where the curve $\partial \Delta$ is oriented counterclockise, i.e. it is oriented with respect to the sector $\Delta.$
Rewrite the latter formula as
\begin{equation*}
2\pi i\cdot f(z)=\int_{\partial \Delta} \frac{f(\zeta) e^{p(\zeta-z)}}{\zeta-z}d\zeta,
 \quad \text{for }z\in \Delta.
\end{equation*}
 Split integration over the oriented curve $\partial{\Delta}$ into integration over its two rays: $e^{-i\alpha}[0,+\infty)$ and $e^{i\alpha}[0,+\infty),$ to get
\begin{equation}\label{cauchy_integral_formula}
2\pi i\cdot f(z)=
\int_{e^{-i\alpha}[0,+\infty)}
\frac{f(\zeta) e^{p(\zeta-z)}}{\zeta-z}d\zeta
-
\int_{e^{i\alpha}[0,+\infty)}
 \frac{f(\zeta) e^{p(\zeta-z)}}{\zeta-z}d\zeta,
 \quad \text{for }z\in \Delta.
\end{equation}
\section{ Fourier inversion.}
Note the following elementary formula that recovers the Cauchy kernel $1/z$ from exponential functions $z\rightarrow e^{\omega z}$ by a proper integration over parameter $\omega:$ \begin{equation}\label{cauchy_to_fourier}
    \frac{1}{z}=-\int_0^{+\infty}e^{\omega z}d\omega,\quad\text{for }Re(z)<0.
\end{equation}
Now let $\zeta\in e^{i\alpha}[0,+\infty)$ and $z\in \Delta.$ Then, by definition \eqref{s} of the sector $\Delta$ and restriction \eqref{alpha} on the angle $\alpha,$ we have $\arg(\zeta-z)\in (\alpha,\alpha+\pi).$ Hence, $\pi/2-\alpha+\arg(\zeta-z)\in (\pi/2,3\pi/2),$ or in other words
\begin{equation*}
    Re\left(ie^{-i\alpha}(\zeta-z)\right)<0.
\end{equation*}
Hence, similarly to formula \eqref{cauchy_to_fourier}, we have
\begin{equation}\label{cauchy_to_fourier_1}
    \frac{e^{p(\zeta-z)}}{\zeta-z}=-\int_{p+ie^{-i\alpha}[0,+\infty)}e^{\omega(\zeta-z)}d\omega
    ,\quad \text{for }\zeta\in e^{i\alpha}[0,+\infty),\;z\in \Delta.
\end{equation}
By an analogous argument we can also prove that
\begin{equation}\label{cauchy_to_fourier_2}
\frac{e^{p(\zeta-z)}}{\zeta-z}=-\int_{p-ie^{i\alpha}[0,+\infty)}e^{\omega (\zeta-z)}d\omega
,\quad\text{ for } \zeta\in e^{-i\alpha}[0,+\infty),\; z\in \Delta.
\end{equation}
Formulas \eqref{cauchy_to_fourier_1} and \eqref{cauchy_to_fourier_2} allow us to rewrite Cauchy's integral formula \eqref{cauchy_integral_formula} as:
\begin{align}\label{f_prefubini}
2\pi i f(z)=
-
\int_{e^{-i\alpha}[0,+\infty)}f(\zeta)
\int_{p-ie^{i\alpha}[0,+\infty)}
e^{\omega (\zeta-z)}d\omega
d\zeta
+
\nonumber\\
+
\int_{e^{i\alpha}[0,+\infty)}f(\zeta)
\int_{p+ie^{-i\alpha}[0,+\infty)}
e^{\omega (\zeta-z)}d\omega
d\zeta,
\quad \text{for }z\in \Delta.
\end{align}
The following two lemmas justify changing the order of integration in \eqref{f_prefubini}:
\begin{lemma}\label{g_estimate}
Under assumptions of theorem \ref{spw} we have
\begin{equation*}
\int_{e^{-i\alpha}[0,+\infty)}\left|f(\zeta)\right|\left|e^{\omega\zeta}\right||d\zeta|
\leq
\frac{-C_{\epsilon}}{h+\epsilon+p\cos(\alpha)}
    ,\quad \text{for }\omega\in p-ie^{i\alpha}[0,+\infty).
\end{equation*}
\end{lemma}
\begin{proof}
Denote
\begin{equation}\label{omega_prime}
    \omega^{\prime}=p-\omega.
\end{equation}
Then conditions 
$\omega\in p-ie^{i\alpha}[0,+\infty)$ 
and 
$\zeta\in e^{-i\alpha}[0,+\infty)$
imply that
\begin{equation}\label{omega_prime_property}
    \left|e^{-\omega^{\prime}\zeta}\right|=1.
\end{equation}
We estimate
\begin{align*}
&
    \int_{e^{-i\alpha}[0,+\infty)}\left|f(\zeta)\right|\left|e^{\omega\zeta}\right||d\zeta|
\overset{\eqref{f_extended},\eqref{omega_prime},\eqref{omega_prime_property}}{\leq}
C_{\epsilon}\int_{e^{-i\alpha}[0,+\infty)}e^{(h+\epsilon)|\zeta|}e^{ pRe(\zeta)}|d\zeta|=
\\
&
=
C_{\epsilon}\int_{e^{-i\alpha}[0,+\infty)}e^{(h+\epsilon+ p\cos(\alpha))|\zeta|}d|\zeta|.
\end{align*}
From this estimate lemma \ref{g_estimate} follows.
\end{proof}
\begin{lemma}\label{lemma_fubini}
We have
\begin{equation*}
    \int_{p-ie^{i\alpha}[0,+\infty)}
\left|e^{-\omega z}\right|
\left|d\omega\right|=\frac{-\left|e^{-pz}\right|}{Re\left(ie^{i\alpha}z\right)},\quad\text{for }z\in\Delta.
\end{equation*}
\end{lemma}
\begin{proof}
We evaluate
\begin{equation*}
 \int_{p-ie^{i\alpha}[0,+\infty)}
\left|e^{-\omega z}\right|
\left|d\omega\right|=\frac{-\left|e^{-pz}\right|}{Re\left(ie^{i\alpha}z\right)}<+\infty,\quad\text{for }-\alpha<\arg(z)<\pi-\alpha.
\end{equation*}
At the same time, due to \eqref{alpha}, condition $-\alpha<\arg(z)<\pi-\alpha$ holds for all $z\in \Delta.$
\end{proof}
By lemmas \ref{g_estimate} and \ref{lemma_fubini}, the  condition of Fubini's theorem for changing the order of integration in the first term of \eqref{f_prefubini} holds. Similarly, one can check that the condition of Fubini's theorem for changing the order of integration in the second term of \eqref{f_prefubini} holds, too.
We apply Fubini's theorem to both terms of \eqref{f_prefubini} to get the following Fourier inversion formula:
\begin{align}\label{fourierinversion}
2\pi i f(z)=-
\int_{p-ie^{i\alpha}[0,+\infty)}
\left(
\int_{e^{-i\alpha}[0,+\infty)}
f(\zeta)
e^{\omega \zeta}
d\zeta
\right)
e^{-\omega z}
d\omega
+
\nonumber\\
+
\int_{p+ie^{-i\alpha}[0,+\infty)}
\left(
\int_{e^{i\alpha}[0,+\infty)}
f(\zeta)
e^{\omega \zeta}
d\zeta
\right)
e^{-\omega z}
d\omega.
\end{align}
\section{ Laplace transform.}
We have the following estimate for the indicator:
\begin{equation*}
    I_f(\theta)=\limsup_{s\rightarrow +\infty}\frac{\ln\left|f\left(se^{i\theta}\right)\right|}{s}\overset{\eqref{f_extended}}{\leq}h,\quad\text{for }|\theta|\leq \alpha.
\end{equation*}
Hence the integral in \eqref{g_directional} is absolutely convergent, and the function $g_{\theta}$ is well-defined on the set $\Omega_{\theta}.$ Further, from Morera's theorem it follows that $g_{\theta}$ is holomorphic in $\Omega_{\theta},$
\begin{equation}\label{g_directionalholomorphic}
    g_{\theta}\in Hol\left(\Omega_{\theta}\right).
\end{equation}
\begin{lemma}\label{well-defined}
Under notations of theorem \ref{spw}, the concatenated Laplace transform $g$ is well-defined by formula \eqref{g}: that is, in \eqref{g} the choice of a particular direction $\theta$ satisfying condition $\omega\in \Omega_{\theta}$ is irrelevant. 
\end{lemma}
\begin{proof}
Let 
\begin{align}
\label{concatenated_theta}
&
-\alpha\leq \theta_1\leq \theta\leq \theta_2\leq\alpha,
\\
\label{concatenated_omega}
&
\omega(s)=-se^{-i\left(\theta_1+\theta_2\right)/2},
\\
\label{concatenated_s}
&
s\geq 0.
\end{align}
We estimate
\begin{align}\label{concaten}
&
    Re\left(\omega(s) e^{i\theta}\right)
    \overset{\eqref{concatenated_omega},\eqref{concatenated_s}}{=}
    -s  Re\left( e^{i(\theta-\theta_2/2-\theta_1/2)}\right)=
    \\\nonumber &
    -s  \cos(\theta-\theta_2/2-\theta_1/2)
    \overset{\eqref{concatenated_s},\eqref{concatenated_theta},\eqref{alpha}}{\leq}
    -s\cos\left(\alpha\right).
\end{align}
By taking $\theta=\theta_1$ and $\theta=\theta_2$ in the estimate \eqref{concaten}, due to definition \eqref{omega_directional} we have
\begin{equation*}
\omega(s)\in \Omega_{\theta_1}\cap \Omega_{\theta_2} ,\quad \text{for }s>>1.
\end{equation*}
For $ \arg(\zeta)\in \left[\theta_1,\theta_2\right]$ we estimate
\begin{align}\label{fomegazeta}
&
    \left|f(\zeta)e^{\omega(s)\zeta}\right|
    \overset{\eqref{f_extended}}{\leq}
    C_{\epsilon}e^{(h+\epsilon)|\zeta|+Re(\omega(s)\zeta)}=
    \\\nonumber &
=
      C_{\epsilon}e^{
      \left(
      h+\epsilon+Re\left(\omega(s)e^{i\arg(\zeta)}\right)
      \right)
      |\zeta|
      }
\overset{\eqref{concaten}}{\leq}
    C_{\epsilon}e^{
      \left(
      h+\epsilon-s\cos(\alpha)
      \right)
      |\zeta|
      }.
           \end{align}
      For $s>>1$ the coefficient       $h+\epsilon-s\cos(\alpha)$ in front of $|\zeta|$ in the last exponent of \eqref{fomegazeta} is negative.
Consequently, the Phragmen-Lindeloef maximum principle applies to the Cauchy integral theorem for the function $\zeta\rightarrow f(\zeta)e^{\omega(s)\zeta}$ in the sector $\lbrace \zeta\in \mathbb{C}\colon \arg(\zeta)\in \left[\theta_1,\theta_2\right]\rbrace,$ 
\begin{equation*}
\int_{e^{i\theta_1}[0,+\infty)}
f(\zeta)
e^{\omega(s) \zeta}
d\zeta=
\int_{e^{i\theta_2}[0,+\infty)}
f(\zeta)
e^{\omega(s) \zeta}
d\zeta.
\end{equation*}
Or equivalently
\begin{equation}\label{g_well_defined}
    g_{\theta_1}(\omega(s))=g_{\theta_2}(\omega(s)),\quad s>>1.
\end{equation}
Hence by uniqueness of analytic functions,
\begin{equation}\label{gwelldefined}
      g_{\theta_1}(\omega)=g_{\theta_2}(\omega),\quad \omega\in \Omega_{\theta_1}\cap \Omega_{\theta_2}.
\end{equation}
\end{proof}
Due to claim \eqref{g_directionalholomorphic} and lemma \ref{well-defined}, we have
\begin{equation}\label{gholomorphic}
    g\in Hol(\Omega).
\end{equation}
\begin{remark}\label{fasg}
We can rewrite the Fourier inversion formula \eqref{fourierinversion} in terms of the concatenated Laplace transform $g$ (defined by \eqref{g} and \eqref{g_directional})) and in terms of the curve $\Gamma$ (parametrized by \eqref{gamma}) as formula \eqref{morimoto}.
\end{remark}
\begin{remark}\label{g_estimate_all}
By \eqref{gamma} the curve $\Gamma$ consists of two subsets: $p-ie^{i\alpha}[0,+\infty)$ and
$p+ie^{i\alpha}[0,+\infty).$
Lemma \ref{g_estimate} provides an estimate on the function $g$ along the subset $p-ie^{i\alpha}[0,+\infty).$
One can prove a similar estimate for the other subset of $\Gamma$ in order to obtain the following unified estimate:
\begin{equation*}
|g(\omega)|
\leq
\frac{-C_{\epsilon}}{h+\epsilon+p\cos(\alpha)}
    ,\quad \text{for }\omega\in \Gamma.
\end{equation*}
\end{remark}
\section{ Domain of analyticity.}\label{nda_proof}
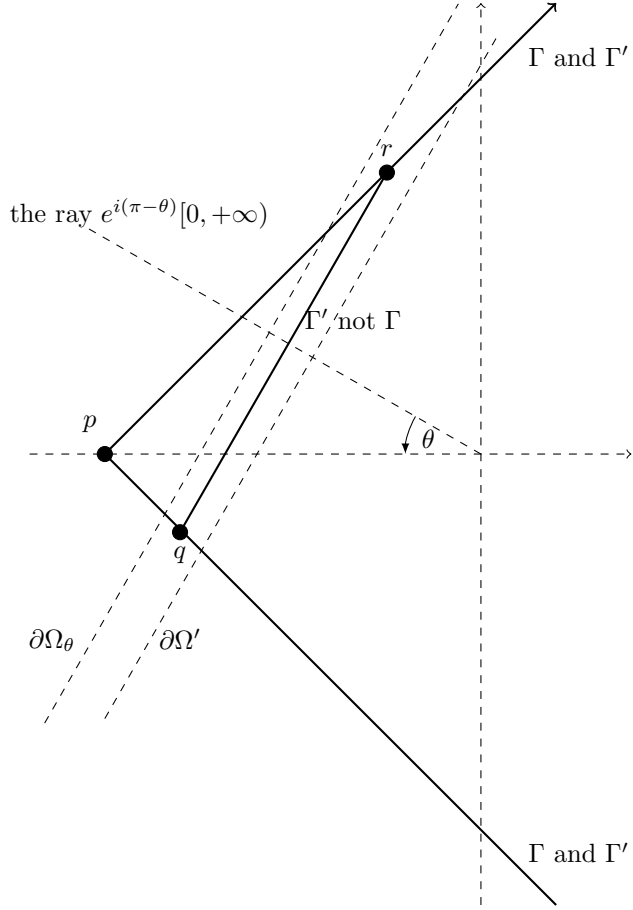
\begin{figure}\label{fig_gamma_prime}
\begin{tikzpicture}[scale=1]
\draw[dashed,->] (-6,0) -- (2,0); 
\draw[dashed,->] (0,-6) -- (0,6); 
\draw[fill=black] (-5,0) circle (0.1cm); 
\draw[fill=black] (-8+10*0.4,-8+17.4*0.4) circle (0.1cm); 
\draw[fill=black] (-8+10*0.675,-8+17.4*0.675) circle (0.1cm); 
\node at (-8+10*0.4,-8+17.4*0.4-0.3) {$q$}; 
\node at (-8+10*0.675,-8+17.4*0.675+0.3){$r$}; 
\draw[thick,->] (-5,0) -- (1,6);
\draw[thick] (-5,0) -- (1,-6);
\node at (-5.2,0.4) {$p$}; 
\node at (1.3,5.3) {$\Gamma$ and $\Gamma^{\prime}$};
\node at (1.3,-5.3) {$\Gamma$ and $\Gamma^{\prime}$};
\draw[dashed](0,0)-- (-0.87*6,0.5*6);
\node at (-0.87*5.8+0.5,0.5*5.8+0.3){the ray $e^{i(\pi-\theta)}[0,+\infty)$};
\draw[-latex] (-0.87,0.5) arc[start angle=150, end angle=180,radius=1cm] ;
\node at (-0.7,0.2) {$\theta$};
\node at (-1.7,1.8) {$\Gamma^{\prime}$ not $\Gamma$};
\draw[dashed](-8.8+10*.3,-8.8+17.4*.3)-- (-8.8+10*.85,-8.8+17.4*.85);
\node at (-9+10*.3+0.3,-9+17.4*.3+1.3){$\partial\Omega_{\theta}$}; 
\draw[dashed](-7+10*.2,-7+17.4*.2)-- (-7+10*.72,-7+17.4*.72);
\node at (-6+10*.12+0.8,-5+17.4*.12+0.45){$\partial\Omega^{\prime}$}; \draw[thick](-8+10*0.4,-8+17.4*0.4)-- (-8+10*0.675,-8+17.4*0.675);
\end{tikzpicture}
\caption{ The curve $\Gamma^{\prime}.$}
\end{figure}
Under assumptions of theorem \ref{nda}, suppose that the function $g$ may be analytically continued to a half-plane
$\Omega^{\prime}$ that contains the half-plane $\Omega_{\theta}$ properly,
$\Omega_{\theta}\subsetneq \Omega^{\prime}.$ We will show that this supposition leads to contradiction \eqref{contradiction}.

Indeed, on this supposition we can define the curve $\Gamma^{\prime}$ as in illustration \ref{fig_gamma_prime} on page \pageref{fig_gamma_prime}. That is, the difference between curves $\Gamma$ and $\Gamma^{\prime}$ is that the curve $\Gamma^{\prime}$ passes directly from point $q$ to point $r$ instead of passing from $q$ to $p$ and then from $p$ to $r,$ like the curve $\Gamma$ does.

The curve $\Gamma^{\prime}$ is bounded away from the half space $\Omega_{\theta}.$ Hence the set 
\begin{equation*}
\lbrace Re\left(\omega e^{i\theta}\right) \colon \omega\in\Gamma^{\prime}\rbrace,
\end{equation*}
being the projection of the curve $\Gamma^{\prime}$ onto the ray $e^{i(\pi-\theta)}[0,+\infty),$ is bounded away from $-I_f(\theta),$
  \begin{equation}\label{gammaprime}
  \inf_{\omega\in\Gamma^{\prime}}Re\left(\omega e^{i\theta}\right)>    -I_f(\theta).
  \end{equation}
  Inequality \eqref{gammaprime} that may not hold for the original curve $\Gamma.$
  
Since in theorem \ref{nda} the function $g$ is assumed to be analytic in the whole half-plane $\Omega^{\prime},$ the Cauchy integral theorem applies in $\Omega^{\prime},$ and consequently the Fourier inversion formula \eqref{morimoto} stays valid if in that formula we substitute the curve $\Gamma$ by the curve $\Gamma^{\prime},$  
 \begin{equation*}
    f(z)=\int_{\Gamma^{\prime}}g(\omega)e^{-\omega z }d\omega,\quad z\in \Delta.
\end{equation*}
We split the curve $\Gamma^{\prime}$ into three subcurves: $[q,r],$ $r+ie^{-i\alpha}[0,+\infty)$ and
$q-ie^{i\alpha}[0,+\infty),$
and estimate
\begin{align}\label{j123}
&
    \left|f\left(se^{i\theta}\right)\right|
    \leq
    \int_{\Gamma^{\prime}}
    \left|e^{-se^{i\theta} \omega}\right|
    \left|g(\omega)\right|
    |d\omega|
    =
    \int_{[q,r]}
    \left|e^{-se^{i\theta} \omega}\right|
    \left|g(\omega)\right|
    |d\omega|
    +
    \\\nonumber
    &
    +
    \int_{r+ie^{-i\alpha}[0,+\infty)}
    \left|e^{-se^{i\theta} \omega}\right|
    \left|g(\omega)\right|
    |d\omega|
     +
    \int_{q-ie^{i\alpha}[0,+\infty)} 
    \left|e^{-se^{i\theta} \omega}\right|
    \left|g(\omega)\right|
    |d\omega|=
      \\\nonumber 
      &
      =J_1(s)+J_2(s)+J_3(s),
  \quad\text{for }s> 0.
\end{align}
We further estimate
\begin{equation*}
 J_1(s)\leq  
 e^{-s\inf_{\omega\in\Gamma^{\prime}}Re\left(\omega e^{i\theta}\right)}
    \int_{[q,r]}
       \left|g(\omega)\right|
    |d\omega|,  \quad\text{for }s> 0.
\end{equation*}
So that
\begin{equation}\label{j1}
\limsup_{s\rightarrow+\infty}\frac{\ln\left|J_1(s)\right|}{s}\leq -\inf_{\omega\in\Gamma^{\prime}}Re\left(\omega e^{i\theta}\right).
\end{equation}
We also estimate
\begin{align*}&
      J_3(s)
    \overset{\ref{g_estimate_all}}{\leq}
    \frac{-C_{\epsilon}}{h+\epsilon+p\cos(\alpha)}
     \int_{q-ie^{i\alpha}[0,+\infty)} 
            \left|e^{\omega\left(\zeta-se^{i\theta}\right)}\right|
    |d\omega|
    \leq
    \\\nonumber &
    \leq
    \frac{-C_{\epsilon}}{h+\epsilon+p\cos(\alpha)}\frac{- \left|e^{-qse^{i\theta}}\right|}{Re\left(ie^{i\alpha}se^{i\theta}\right)},  \quad\text{for }s> 0
\end{align*}
(The last inequality is by repeating the arguments of lemma \ref{lemma_fubini} for $q$ instead of $p$).

So that
\begin{equation}\label{j3}
    \limsup_{s\rightarrow+\infty}\frac{J_3(s)}{s}\leq -Re\left(q e^{i\theta}\right)=-\inf_{\omega\in\Gamma^{\prime}}Re\left(\omega e^{i\theta}\right).
\end{equation}
Similarly, one can prove that
\begin{equation}\label{j2}
    \limsup_{s\rightarrow+\infty}\frac{J_2(s)}{s}\leq -Re\left(r e^{i\theta}\right)=-\inf_{\omega\in\Gamma^{\prime}}Re\left(\omega e^{i\theta}\right).
\end{equation}
Using the estimates \eqref{j1},\eqref{j2},\eqref{j3} in \eqref{j123}, we get
\begin{equation}\label{contradiction}
    I_f(\theta)\leq  -\inf_{\omega\in \Gamma^{\prime}
    }Re\left(\omega e^{i\theta}\right)\overset{\eqref{gammaprime}}{<}I_f(\theta).
\end{equation}
A contradiction.

\end{document}